\documentclass{amsart}

\title{Smallest non-cyclic quotients of the automorphism group of free groups}
\author{Sudipta Kolay}
\address{ICERM, 121 South Main Street, 11th Floor\\Providence, RI 02903, USA.}
\email{sudipta\_kolay@brown.edu}
\date{}

\usepackage[a4paper,top=2.5cm,bottom=2cm,left=3cm,right=3cm,marginparwidth=1.5cm]{geometry}
\usepackage{amsmath,amsthm,amsfonts,amssymb}
\usepackage{graphicx}
\usepackage[colorlinks=true, allcolors=blue]{hyperref}
\usepackage{enumitem} 
\usepackage{dirtytalk}

\theoremstyle{plain}
\newtheorem{thm}{Theorem}

\newtheorem{prop}[thm]{Proposition}

\newtheorem{lem}[thm]{Lemma}

\theoremstyle{definition}

\theoremstyle{remark}

\newtheorem{example}[thm]{Example}

\begin{document}

\begin{abstract}
    We give an new, elementary proof of the result that the smallest non-cyclic quotients of automorphism group of free group is the linear group over the field of two elements, and moreover all minimal quotients are obtained by the standard projection composed with an automorphism of the image. This result, originally due to  Baumeister--Kielak--Pierro, proves a conjecture of Mecchia--Zimmermann.
\end{abstract}

\maketitle

\section{Introduction}
In this paper we explore the smallest non-cyclic quotients of $Aut(F)$, the automorphism group of a finite rank free group $F$. This is a particular instance of the following broad problem stated by Roger Lyndon~\cite{Lyndon1987PROBLEMSIC} in 1977:
\begin{center}
    \say{Determine the structure of Aut($F$), of its subgroups, especially its finite subgroups, and its quotient groups, as well as
the structure of individual automorphisms.}
\end{center}

Mecchia--Zimmermann \cite{Mecchia2009OnMF} proved that for $n\in\{3,4\}$, the the  smallest  non-trivial (respectively non-abelian) quotient  of SOut($F_n$) (respectively Out($F_n$)) is $SL(n,\mathbb{Z}_2)$,  and conjectured the same statement holds for arbitrary $n\geq 3$. This conjecture was proved by Baumeister--Kielak--Pierro \cite{Baumeister2019OnTS}, 
using the classification of finite simple groups and representation theory of SAut($F_n$). In fact they proved a stronger result for SAut($F_n$) and Aut($F_n$). Our goal in this paper is to give an alternate, elementary proof of this result, using the inductive orbit stabilizer method \cite[Section 3]{Kolay2021BMCG}.

\begin{thm}\label{A} Suppose $n\geq 3$, and $G$ be any of the groups $\mathrm{Aut}(F_n)$, $\mathrm{SAut}(F_n)$, $\mathrm{Out}(F_n)$, $\mathrm{SOut}(F_n)$, $GL(n,\mathbb{Z})$, $SL(n,\mathbb{Z})$.
If $H$ is a non-cyclic quotient of $G$, then either $|H|> |SL(n,\mathbb{Z}_2)|$, or $H$ is isomorphic to $SL(n,\mathbb{Z}_2)$. Moreover, in the latter case the quotient map $G\rightarrow H$ is obtained by postcomposing the natural map $\pi$ with an automorphism of $SL(n,\mathbb{Z}_2)$. 
\end{thm}

 We state our result for non-cyclic quotients, instead of non-abelian or non-trivial quotients, but they are equivalent for $n\geq 3$. The groups Aut$(F_n)$, Out$(F_n)$, and $GL(n,\mathbb{Z})$ all have abelianizations $\mathbb{Z}_2$, and so their smallest non-abelian quotient is same as smallest non-cyclic quotient. Similarly, the groups SAut$(F_n)$, SOut$(F_n)$, and $SL(n,\mathbb{Z})$ are perfect (i.e.  abelianization is trivial), and so their smallest non-trivial quotient is same as smallest non-cyclic quotient.
 
 Theorem~\ref{A} does not hold in case $n=2$ for the groups $\mathrm{Aut}(F_2)$, $\mathrm{Out}(F_2)$, $GL(2,\mathbb{Z})$, where the smallest non-cyclic quotient is the Klein's four group $\mathbb{Z}_2\oplus \mathbb{Z}_2$. However, as we shall see, Theorem~\ref{A} is also true for the groups $\mathrm{SAut}(F_2)$, $\mathrm{SOut}(F_2)$, $SL(2,\mathbb{Z})$.

Let us note that it suffices to prove the above theorem for the family SAut($F_n$), because:
 \begin{enumerate}
     \item Any quotient of Aut($F_n$) can be restricted to obtain a quotient of SAut($F_n$), and if this quotient is trivial, then the original quotient must be cyclic since Aut($F_n)/$SAut$(F_n)\cong \mathbb{Z}_2$.
     \item If $A\rightarrow B$ and $B\rightarrow C$ are group epimorphisms, then if $C$ is the smallest quotient of $A$, then it is also the smallest quotient of $B$.
 \end{enumerate}

 The situation closely parallels that of braid groups $B_n$ and mapping class groups $\mathrm{Mod}(\Sigma_g)$. The quotients of smallest orders in these cases are symmetric groups $S_n$ and symplectic groups $\mathrm{Sp}(2g,\mathbb{Z}_2)$, respectively, and we have similar results for the quotient maps \cite{Kolay2021BMCG}.  Zimmermann\cite{Zimmermann2008ANO} proved the result for mapping class groups in the special cases $g\in\{3,4\}$, and conjectured the same result for all higher $g$. This conjecture was proved by Kielak--Pierro \cite{Kielak2017OnTS}, using very similar techniques as that of Baumeister--Kielak--Pierro \cite{Baumeister2019OnTS} result mentioned earlier. The author gave an elementary proof of these results for both braid and mapping class groups, using the inductive orbit-stabilizer method; and this paper is an analogue for the setting of automorphism group of free groups. Although the most natural analogue would have been to prove an optimal lower bound on the orbit of free transvections under any non-cyclic quotient, we were unable to prove this directly. However, we can make a similar approach work by looking at orbit of subgroups fixing a hyperplane setwise; illustrating the robustness of the inductive orbit-stabilizer method.\\

\noindent\emph{Acknowledgments}: The author is grateful to Dan Margalit and John Etnyre for helpful discussions. The author thanks Nancy Scherich for comments on an earlier draft of this paper. This work is supported by NSF grants DMS-1439786 and DMS-1906414 while the author is/was located at ICERM and Georgia Tech respectively.\\

\section{Setup}
The goal of this section is to set up appropriate notation and prove some preliminary results.

 \subsection{Background.}
We will use the notation Aut$(F_n)$ (respectively SAut$(F_n)$) to denote the group of all (respectively special\footnote{by special we mean orientation preserving, i.e. if we compose with the natural projection to $GL(n,\mathbb{Z})$, we should only get special (determinant 1) linear maps.}) automorphisms of the free group $F_n$.
Since the abelianization of $F_n$ is $\mathbb{Z}^n$, every automorphism of $F_n$ induces an automorphism of $\mathbb{Z}^n$. By further composing with mod $m$ reduction (for any natural number $m$), and subsequently modding out by the center, we obtain group homomorphisms\footnote{however, these will not be surjective in general, as any matrix in $GL(n,\mathbb{Z})$ has determinant $\pm 1$, but there can be more units in $\mathbb{Z}_m$ (this issue does not arise for the $SL(n,\mathbb{Z})$  setting).}:
$$\mathrm{Aut}(F_n)\rightarrow  GL(n,\mathbb{Z})\rightarrow  GL(n,\mathbb{Z}_m)\rightarrow PGL(n,\mathbb{Z}_m).$$
Analogously, we obtain the following for special automorphisms:
$$\mathrm{SAut}(F_n)\rightarrow  SL(n,\mathbb{Z})\rightarrow  SL(n,\mathbb{Z}_m)\rightarrow PSL(n,\mathbb{Z}_m).$$
These gives various non-cyclic quotients of (special) automorphism of free groups, and our main theorem states that the minimal quotient is obtained with $m=2$ above. Let us note that for $m=2$, we have the isomorphisms $GL(n,\mathbb{Z}_2)=PGL(n,\mathbb{Z}_2)=SL(n,\mathbb{Z}_2)=PSL(n,\mathbb{Z}_2)$.

 \subsection{Hyperplanes.}\label{hyp}
 Let us fix a basis $x_1,...,x_{n}$ of $F_{n}$ (and we will use this notation implicitly throughout this paper), and by applying the abelianization  map we get a basis $e_1,...,e_{n}$ of $\mathbb{Z}^{n}$. Reducing modulo 2, we obtain a basis $\epsilon_1,...,\epsilon_{n}$ of $\mathbb{Z}_2^{n}$.\\
 \noindent\textit{Hyperplanes in $\mathbb{Z}_2^{n}$}:
 We note that any hyperplane $P$ of $\mathbb{Z}_2^{n}$ can be characterized by an indicator vector $(i_1,...,i_n)$ where $i_j$ equals 1 if $\epsilon_j\in P$, and 0 otherwise. Conversely, given any binary vector, with the exception of the all ones vector\footnote{which corresponds to all of $\mathbb{Z}_2^{n}$.}, there is a unique hyperplane which has that as the indicator vector, as explained below. Firstly, let us note that if $\epsilon_j$ and $\epsilon_k$ (with $j\neq k$) are both not in a hyperplane $P$ of $\mathbb{Z}_2^{n}$, then $P$ must contain $\epsilon_j+\epsilon_k$ as any hyperplane has exactly two cosets.  Thus, given a binary vector $I=(i_1,...,i_n)$ different from the all ones vector,  we get the corresponding hyperplane $P_I$ determined by the equation $ \sum_{j=1}^n (1-i_j)\epsilon_j=0$. We can also describe $P_I$ in terms of a basis, as follows. If the indices (say, written in increasing order) $j$ so that $i_j=0$ are $z_1,\cdots,z_l$ and the remaining indices are $o_1,\cdots,o_{n-l}$, a basis of $P_I$ is given by $$\epsilon_{z_1}+\epsilon_{z_2},\cdots, \epsilon_{z_{l-1}}+\epsilon_{z_{l}}, \epsilon_{o_1},\cdots, \epsilon_{o_{n-l}}$$
 
 \noindent\textit{Hyperplanes in $F_{n}$}:
With the same notation as above, we can define the subgroup $S_I$ of $F_n$ with free basis  $$x_{z_1}x^{-1}_{z_2},\cdots, x_{z_{l-1}}x^{-1}_{z_{l}}, x_{o_1},\cdots, x_{o_{n-l}},$$
and we note that $S_I$ projects precisely to $P_I$. We note that we can extend it to a basis of all $F_n$ by adding any of the $x_j$ which is not in $S_I$ (i.e., add any $x_j$ where $i_j=0$).\\
The following lemma will be useful for us later.
\begin{lem}\label{lemA}
If $I$ and $I'$ are two distinct binary indicator vectors (neither equaling the all ones vector), we can choose a new basis of $F_n$, so that in using the corresponding basis of $\mathbb{Z}_2^{n}$,
the indicator vectors of the hyperplanes $S_I$ and $S_{I'}$ are $(1,0,1,\cdots, 1)$ and $(0,1,1,\cdots, 1)$ respectively.
\end{lem}

Before we prove the lemma, let us illustrate it by an example below.

\begin{example} Suppose $I=(1,1,0,0,0)$ and $I'=(0,0,0,1,1)$. The desired result holds when we consider the new basis of $F_5$ to be $y_1=x_1$, $y_2=x_2x_3^{-1}, y_3=x_1x_2^{-1},y_4=x_1x_3^{-1}x_4$ and $y_5=x_4x_5^{-1}$.
\end{example}

\begin{proof}[Proof of Lemma~\ref{lemA}] Let us begin by considering the special case that $I$ and $I'$ disagree on all entries. Without loss of generality, by changing the order of the free basis, let us assume precisely the first $k$ entries of $I$ are 1, and the rest 0 (and thus the first $k$ entries of $I'$ are 0, and the rest 1). Then the desired result holds in the new basis:
$y_1=x_1, y_2=x_{k+1}, y_{2+j}=x_jx_ {j+1}^{-1} $ for $1\leq j \leq k-1$, and $y_{1+j}=x_jx_ {j+1}^{-1}$ for $k+1\leq j \leq n-1.$\\
If we now consider the case where we append $m$ 1's to both $I$ and $I'$ (in the special case of the previous paragraph) to get binary indicator vectors with $n+m$ entries, we see that the lemma continues to hold if we extend the above basis by setting $y_{n+l}=x_{n+l}$ with $1\leq l \leq m$.

Now let us consider the case where we append $z$ 0's to both $I$ and $I'$ in the previous paragraph. The result holds if we  now extend the basis by setting  $y_{n+m+l}=x_{n+m+l}x_{n+m+l+1}^{-1}$ with $1\leq l \leq z-1$ and $y_{n+m+z}=x_1 x_{n+m+1}^{-1}x_2$.

Lastly, in the general case, we can reorder the basis elements so that we are in the case of the previous paragraph, and so the lemma is proved.

\end{proof}

\subsection{Some facts about automorphism groups of free groups} 
  We mainly follow the conventions of Gersten \cite{Gersten1984APF}, that is, maps in $\mathrm{Aut}(F_n)$ are applied left to right (so $\phi\cdot \psi$ means first apply $\phi$ and then $\psi$). Let us define the free transvection $E_{x_i,x_j}$ by $x_i \mapsto x_ix_j$ and fixing all other $x_k$'s, and $E_{x_i^{-1},x_j}$ by $x_i \mapsto x_j^{-1}x_i$ and fixing all other $x_k$'s.  It turns out the free transvections $E_{x_i^{\pm{1}},x_j}$ generate $\mathrm{SAut}(F_n)$, and they satisfy the relations $[E_{x_i,x_j},E_{x_k,x_l}]=1$ for distinct $i,j,k,l$; and $[E_{x_i,x_j},E_{x_j,x_k}]=E_{x_i,x_k}$.  We refer the reader to \cite{Gersten1984APF} for more details and a presentation of SAut$(F_{n})$.\\
 For future use, let us also record here the fact, about normal generators of the Torelli  group, i.e. the kernel of the homomorphism $\mathrm{Aut}(F_n)\rightarrow  GL(n,\mathbb{Z})$.
 \begin{prop}\label{toreli}\cite[Magnus]{Magnus1935berndimensionaleG} For $n\geq 2$, the Torelli  group is normally generated by $E_{x_1,x_2}E_{x_1^{-1},x_2}$.

 \end{prop}

\subsection{Some subgroups of the group of special automorphisms}

Let us define the subgroup $\mathcal{A}$ of SAut$(F_{n+1})$ to consist of all special automorphisms which fix $x_{n+1}$, and induces an automorphism on the free subgroup $F_n$ generated by $x_1,...,x_n$. Thus $\mathcal{A}$ is isomorphic to SAut($F_n$).
Similarly, we define the subgroup $\mathcal{B}$ of SAut$(F_{n+1})$ to be the embedded copy $F_n\times F_n$ defined by $$(v,w)\mapsto \phi_{v,w},\text{ where $\phi_{v,w}$ is defined by } [x_1,....,x_n,x_{n+1}]\mapsto [x_1,....,x_n,vx_{n+1}w^{-1}]. $$

Similarly, we can define subgroups  $\mathcal{B}_1$ and $\mathcal{B}_2$ to be the embedded copies of $F_n$ defined by the formulas  $v\mapsto \phi_{v,1}$, and  $w\mapsto \phi_{1,w}$ respectively (i.e. we are looking at the image of the factors $F_n\times \{1\}$ and $\{1\}\times F_n$ in $F_n\times F_n$).

Let us note that for any $\psi\in \mathcal{A}$, we have $\psi^{-1} \phi_{v,w}\psi= \phi_{\psi(v),\psi(w)}$, and thus $\psi^{-1}\mathcal{B}\psi= \mathcal{B}$. Consequently, we have $\mathcal{A}\mathcal{B}=\mathcal{B}\mathcal{A}$; and therefore $\mathcal{C}:=\mathcal{A}\mathcal{B}$ is a subgroup of SAut$(F_{n+1})$. Moreover we see that $\mathcal{B}$, $\mathcal{B}_1$  and $\mathcal{B}_2$ are all normal subgroups of $\mathcal{C}$.

\section{Proof of Theorem~\ref{A}}

We will prove Theorem~\ref{A} by induction, and we begin by stating the induction hypothesis.

\subsection{Induction hypothesis} For any natural number $n\geq 2$, let $P(n)$ be the statement:\\ If $H$ is a non-cyclic quotient of $\mathrm{SAut}(F_n)$, then either $|H|> |SL(n,\mathbb{Z}_2)|$, or $H$ is isomorphic to $SL(n,\mathbb{Z}_2)$. Moreover, in the latter case the quotient map $\mathrm{SAut}(F_n)\rightarrow H$ is obtained by postcomposing the natural map $\pi$ with an automorphism of $SL(n,\mathbb{Z}_2)$.\\

\noindent We will prove Theorem~\ref{A} by (checking the base cases and) carrying out the following steps of the inductive orbit stabilizer method.
\begin{enumerate}
    \item The conjugacy class of the image of $\mathcal{C}$ under any non-cyclic quotient of SAut$(F_{n+1})$ has at least $2^{n+1}-1$ elements.
    \item Inductively, the stabilizer of the image of $\mathcal{C}$ has cardinality at least $2^n\cdot|SL(n,\mathbb{Z}_2)| $, and the result follows by the orbit stabilizer theorem.
\end{enumerate}

\subsection{Base cases} In case $n=2$, the candidate quotient $SL(2,\mathbb{Z}_2)\cong S_3$, which is the smallest non-abelian group. Since SAut($F_2$) has abelianization\footnote{ This fact is a consequence of the presentation of $\mathrm{SAut}(F_2)$, due to Gersten~\cite[Theorem 1.4]{Gersten1984APF}, as explained below. We can check that four of the generators $E_{x_2,x_1}$, $E_{x^{-1}_1,x_2 }$,$E_{x_1,x^{-1}_2 }$ and $E_{x^{-1}_2,x^{-1}_1}$ are conjugate (and the other four generators are their inverses), and so must map to the same element in the abelianization. The result now follows from the relation $w_{ab}^4=1$. } $\mathbb{Z}_{12}$, it  cannot have any smaller non-cyclic quotient (i.e. we are ruling out the Klein's four group). The second statement about the maps also follows readily, since the images of $E_{x_1,x_2}$ and $E_{x_2,x^{-1}_1}$ must be conjugate an non-commuting (and the only pair of such elements in $S_3$ are transpositions).\\
We will hereafter assume $n+1\geq 3$, and later in the inductive step we will show $P(n+1)$ holds assuming $P(n)$ is true. But we will need some preliminary results before that.

\subsection{Lower bound on size of subgroup quotients} Suppose we have a non-cyclic quotient $H$ of SAut$(F_{n+1})$, and let us denote the images of $\mathcal{A}$ and $\mathcal{B}$ by $A$ and $B$ respectively.
As $H$ is non-cyclic, it follows from the Gersten presentation of SAut$(F_n)$ that $A$ must not be a non-cyclic (as the free transvections are all conjugate) quotient of SAut$(F_n)$, and thus we can inductively bound the size of $A$. So we will now focus on finding a lower bound on the size of $B$. Let us first note that our candidate smallest quotient of SAut$(F_{n+1})$ is SL$(n,\mathbb{Z}_2)$, and under this quotient map the image of $\mathcal{B}$ is $\mathbb{Z}_2^n$, which has size $2^n$. So the best possible lower bound for $|B|$ we can hope for is $2^n$, and in fact the following (slightly stronger) lemma shows that this is indeed the case.

\begin{lem}\label{lemB} Under any non-cyclic quotient of SAut$(F_{n+1})$, the image $B$ of $\mathcal{B}$ has size at least $2^n$ (and same statement holds for the images $B_1, B_2$ of $\mathcal{B}_1,\mathcal{B}_2$ respectively).

\end{lem}

\begin{proof} Let $R_j$ denote the free transvection $E_{x_{n+1},x_j}$.
We claim that for $\vec{p}\in\mathbb{Z}_2^n$, the $2^n$ elements $\vec{R}^{\vec{p}}:=R_{1}^{p_1}R_2^{p_2} \cdots R_{n}^{p_n} $ with $p_j\in \{0,1\}$ (with $1\leq j \leq n$) must have different quotient classes. Let us first note that for non-zero $\vec{p}$, the automorphism $\vec{R}^{\vec{p}}$ is a free transvection, and so they must map to non-trivial quotient classes. Now let us suppose that the quotient classes of  $\vec{R}^{\vec{p}}$ and $\vec{R}^{\vec{q}}$ coincide for two non-zero binary vectors $\vec{p}$ and $\vec{q}$. Then by a change of coordinates\footnote{ Let $\vec {X}^{\vec{p}}$ denote $x_n^{p_n}\cdots x_1^{p_1}$. We note that we can extend $\vec {X}^{\vec{p}}$ and $\vec {X}^{\vec{q}}$ to a free basis of $F_n$, by appending $x_i$'s sequentially and throwing out the ones that cause problems. This is analogous to extending $\vec {p}$ and $\vec{q}$ to a vector space basis of $\mathbb{Z}_2^n$ by systematically adding the standard basis vectors $e_i$ and and ignoring the problematic ones.}, we see that we may assume the quotient classes of $R_1$ and $R_2$ coincide. Since we assume $n+1\geq 3$, we see that this implies by the Gersten relations that the quotient classes of $R_3=[R_1,E_{x_1,x_3}]$ and $1=[R_2,E_{x_1,x_3}]$ must coincide. Thus all free transvections are in the kernel, so the quotient must be cyclic, a contradiction. The result follows for $B_2$ (and $B$), and by symmetry it also holds for $B_1$.\end{proof}

 \subsection{The orbit}\label{orb} Using the same notation as above, we will consider the conjugacy class of $AB$ in the non-cyclic quotient $H$. Let us first note that $AB$ must be a proper subgroup of $H$, because if $AB=H$, then we get an even smaller quotient of SAut$(F_{n+1})$ by looking at $H/B_i$ for $i=1,2$ (let us recall from the previous section that $\mathcal{B}_i$ is a normal subgroup of $\mathcal{AB}$, and the fact that images of normal subgroups are normal in the image). If either $H/B_1$ or $H/B_2$ is nontrivial we obtain a
 non-cyclic quotient of strictly smaller cardinality. Otherwise $B_1=H=B_2$, and since $B_1$ and $B_2$ commute, this implies $H$ must be abelian, a contradiction to the fact SAut$(F_{n+1})$ is perfect.
 
\subsection{Bounding the orbit size} We will now show that that there are at least $2^{n+1}-1$ conjugacy classes of $AB$ in $H$ (this corresponds to the subgroups of $SL(n,\mathbb{Z}_2)$ fixing a hyperplane setwise, and we recall $2^{n+1}-1$ hyperplanes in $\mathbb{Z}_2^{n}$).
For any binary indicator vector $I$ (we ignore the all ones vector as always), let us consider the basis of $F_{n+1}$ coming from extending the basis of $S_I$ as we saw in Subsection~\ref{hyp},  and let $\mathcal{A}_I,\mathcal{B}_I,\mathcal{C}_I$ denote the subgroups defined the same way as $\mathcal{A},\mathcal{B},\mathcal{C}$ using this new basis, instead of the original basis $x_1,\cdots,x_{n+1}$ 
(thus the latter three groups are obtained when $I$ equals the $(1,1,\cdots,1,0)$ vector). We note that this group is independent of which $x_j$ we appended to the free basis of $S_I$ to complete it to a basis of $F_{n+1}$. We will denote the image of $\mathcal{C}_I$ by $C_I$. It is clear that all the subgroups $C_I$'s are conjugate (by using a change of coordinates\footnote{and moreover even by a special automorphism by composing with an inversion of a single basis element if required.}) in $H$.

\begin{lem}\label{lemC} For $n\geq 1$, the $2^{n+1}-1$ subgroups $C_I$ are all distinct in any non-cyclic quotient $H$ of SAut$(F_{n+1})$.

\end{lem}
\begin{proof}
Suppose not. Then by the change of coordinates in Lemma~\ref{A}, we may assume that $C_I=C_{I'}$, where $I=(1,0,1,\cdots, 1)$ and $I'=(0,1,1,\cdots, 1)$. But this implies $C_I$ contains the quotient classes of all free transvections in $\mathcal{C}_I$. Moreover, since it also equals $C_{I'}$, we see that $\mathcal{C}_I$ also contains the
quotient classes of $E_{x_1,x_3}$, $E_{x_1^{-1},x_3}$, $E_{x_3,x_1}$, $E_{x_3^{-1},x_1}$. It is now clear from the Gersten presentation of SAut$(F_{n+1})$ that $C_I$ must be $H$, the image of all of SAut$(F_{n+1})$.  This contradicts our discussion in Subsection~\ref{orb} above, so the lemma is proved.
\end{proof}

\subsection{Inductive proof}
We are now ready to complete the proof of our main theorem.
\begin{proof}[Proof of Theorem~\ref{A}] It remains to show that for $n+1\geq 3$, $P(n+1)$ is correct assuming $P(n)$ holds (base cases already checked). By the last subsection, we see that the there are at least $2^{n+1}-1$ conjugacy classes of $C$. Also it is clear that the stabilizer contains $C=AB$. We know from previous discussion that for $i=1,2$, the subgroup $B_i$ (image of $\mathcal{B}_i$) is normal in $C$, we see that $K_i:=A\cap B_i $ is a normal subgroup of $A$.

If for some $i$, the quotient $A/K_i$ of SAut$(F_{n+1})$ is not cyclic (or equivalently non-trivial), by the induction hypothesis $P(n)$, we must have 
$|A/K|\geq |SL(n,\mathbb{Z}_2)|.$ Therefore it follows that from the induction hypothesis and Lemma~\ref{lemB} that:
\begin{equation}
  |C|=|A/K||B|\geq |SL(n,\mathbb{Z}_2)|\cdot 2^n.
\end{equation}

Thus by the orbit stabilizer theorem, we have:
\begin{equation}
  |H|= |Orb(C)|\cdot |Stab(C)|\geq (2^{n+1}-1)|\cdot |C| \geq  (2^{n+1}-1)|\cdot|SL(n,\mathbb{Z}_2)|\cdot 2^n =|SL(n+1,\mathbb{Z}_2)|.
\end{equation}

Finally in the case of equality above, we see from the induction hypothesis $P(n)$ that a normal generator of the Torelli group (see Proposition~\ref{toreli}) must be contained in the kernel of the quotient map $q:\mathrm{SAut}(F_{n+1})\rightarrow H$, and also the image of a free transvection has order 2. Thus $q$ must factor through $SL(n,\mathbb{Z}_2)$, and the result follows.

Lastly, it remains to consider the case that $A/K_1$ and $A/K_2$ are both trivial, but this means $A=K_1=K_2$.
The proof is completed by the exact same argument as in Subsection~\ref{orb} for $n>2$.
However, we need a separate argument for the case $n=2$ since SAut($F_2$) is not perfect. Let us denote the quotient class of $f\in\mathrm{SAut}(F_3)$ in $H$ by $\overline{f}$. We see that we must have $\overline{E_{x_1,x_2}}=\overline{E_{x_2,x^{-1}_1}}$, as the corresponding elements are conjugate in SAut($F_2$) and the quotient is abelian. The relations $E_{x_3,x^{-1}_1}=[E_{x_3,x_2},E_{x_2,x^{-1}_1}]$ and $[E_{x_3,x_2},E_{x_1,x_2}]=1$ in SAut($F_3$) implies:
$$\overline{E_{x_3,x^{-1}_1}}=[\overline{E_{x_3,x_2}},\overline{E_{x_2,x^{-1}_1}}]=[\overline{E_{x_3,x_2}},
\overline{E_{x_1,x_2}}]=1;$$
and hence it follows $H$ must be trivial, a contradiction.
\end{proof}

\bibliographystyle{plain}
\bibliography{references}

\begin{thebibliography}{1}

\bibitem{Baumeister2019OnTS}
B.~Baumeister, Dawid Kielak, and E.~Pierro.
\newblock On the smallest non‐abelian quotient of {Aut($F_n$}).
\newblock {\em Proceedings of The London Mathematical Society}, 118:1547--1591,
  2019.

\bibitem{Gersten1984APF}
S.~Gersten.
\newblock A presentation for the special automorphism group of a free group.
\newblock {\em Journal of Pure and Applied Algebra}, 33:269--279, 1984.

\bibitem{Kielak2017OnTS}
Dawid Kielak and E.~Pierro.
\newblock On the smallest non-trivial quotients of mapping class groups.
\newblock {\em Groups, Geometry, and Dynamics}, 14(2):489–512, 2020.

\bibitem{Kolay2021BMCG}
S.~Kolay.
\newblock Smallest non-cyclic quotients of braid and mapping class groups,
  2021.

\bibitem{Lyndon1987PROBLEMSIC}
R.~Lyndon.
\newblock Problems in combinatorial group theory.
\newblock 1987.

\bibitem{Magnus1935berndimensionaleG}
W.~Magnus.
\newblock {\"U}bern-dimensionale gittertransformationen.
\newblock {\em Acta Mathematica}, 64:353--367, 1935.

\bibitem{Mecchia2009OnMF}
M.~Mecchia and B.~Zimmermann.
\newblock On minimal finite quotients of outer automorphism groups of free
  groups.
\newblock {\em Atti Semin. Mat. Fis. Univ. Modena Reggio Emilia}, page
  115–120, 2011.

\bibitem{Zimmermann2008ANO}
Bruno~P. Zimmermann.
\newblock On minimal finite quotients of mapping class groups.
\newblock {\em Rocky Mountain J. Math.}, 42(4):1411--1420, 08 2012.

\end{thebibliography}

\end{document}